\documentclass[12pt]{amsart}
\usepackage{hyperref}
\usepackage{amsmath}
\usepackage{amsthm}
\usepackage{mathrsfs}
\usepackage{amssymb}
\usepackage{tikz-cd}
\usepackage{comment}
\usepackage{mathtools}
\usepackage{manfnt}
\calclayout
\newcommand{\PP}{\mathbb{P}}
\newcommand{\ZZ}{\mathbb Z}

\newcommand{\OO}{\mathcal O}
\newcommand{\OC}{\mathcal C}
\newcommand{\ra}{\rightarrow}
\newcommand{\Spec}{\mathrm{Spec}}

\newcommand{\twopartdef}[4]
{
	\left\{
		\begin{array}{ll}
			#1 & \mbox{if } #2 \\
			#3 & \mbox{if } #4
		\end{array}
	\right.
}
\newtheorem{theorem}{Theorem}[section]

\newtheorem{lemma}[theorem]{Lemma}
\newtheorem{proposition}[theorem]{Proposition}

\theoremstyle{definition}
\newtheorem{question}[theorem]{Question}
\newtheorem*{remark}{Remark}
\newtheorem{example}[theorem]{Example}

\title{Vector bundles on trees of smooth rational curves}
\author{Geoffrey Smith}
\begin{document}
\maketitle
\begin{abstract}
Given a vector bundle $E$ on a tree of smooth rational curves $C$, we give necessary and sufficient conditions for a vector bundle $E'$ on $\PP^1$ to specialize to $E$ on $C$, generalizing the rank 2 case, due to Coskun.    
\end{abstract}
\section{Introduction}\label{introduction}
Let $C$ be a connected nodal curve of arithmetic genus 0 defined over an algebraically closed field $k$. $C$ admits a simple geometric description, as its irreducible components are each isomorphic to $\PP^1$, and its dual graph is a tree; in what follows, we will refer to it as a \emph{tree of smooth rational curves}, following the convention of \cite[II.7.4]{Kol96}. Likewise, vector bundles on $C$ are relatively easy to describe explicitly. Given the normalization $\tilde{C}$ of $C$, a vector bundle on $C$ is specified by a vector bundle on $\tilde{C}$ and the choice of gluing data on the nodes of $C$. That said, it is much less clear what the vector bundles on \emph{families} of trees of smooth rational curves can look like.

In this paper, we investigate the behavior of vector bundles on trees of smooth rational curves. For the most part, we work in the setting of a family of curves $\OC$ over a base scheme $\Delta=\Spec(k[[t]])$. In this setting, the main question we address is the following.

\begin{question}\label{basicQuestion}
Let $C$ and $C'$ be trees of smooth rational curves over $k$, and let $E$ and $E'$ be vector bundles on $C$ and $C'$ respectively. Under what circumstances does $E'$ specialize to $E$? That is, when is there a flat family $\pi:\mathcal{C}\ra \Delta$, where $\Delta:=\Spec(k[[t]])$, and a vector bundle $\mathcal{E}$ on $\OC$ such that $\OC_0\cong C$ and $\mathcal E\vert_{\OC_0}\cong E$, while  $\OC_\eta\cong C\times_k K$ and $\mathcal E\vert_{\OC_\eta}\cong E$ for $\eta$ the generic point?
\end{question}

Coskun considered a similar question, and in \cite[Section 4]{Cos06}  gave conditions under which a surface scroll specializes to a given  tree of scrolls. This closely corresponds with the rank two case of this question. Likewise, Ran studied a variant of this question in \cite{Ran20}, and gives for instance criteria for a vector bundle $E$ on a tree $C$ to deform only to a balanced bundle on $\PP^1$. 

The value of this question largely lies in the theory of jumping curves. Given a vector bundle $E$ on a projective variety $X\subset \PP^n$, in many cases a key step in understanding the behavior of $E$ involves understanding the restriction of $E$ to lines contained in $X$; in this context, a \emph{jumping line} is a line $\ell$ where the splitting type of $E\vert_\ell$ is more special than at the general point. There is a more extensive discussion of jumping lines and their applications to vector bundles on projective space in \cite{OSS11}.

The theory of jumping lines has been partially extended to higher degree curves on varieties, in, for instance, \cite{Cos08}. One obstruction present in the higher-degree case, however, is the fact that rational curves can degenerate to reducible curves; in this setting, an answer to Question \ref{basicQuestion} clarifies what sorts of jumping curve can exist.

It is  clear that $E$ and $E'$ must have the same degree and rank for $E'$ to specialize to $E$. Beyond this, the best-known obstruction to $E'$ specializing to $E$ is a failure of the upper semicontinuity condition for cohomology of coherent sheaves. Every line bundle of degree $d$ on $C$ deforms to $\OO(d)$ on $\PP^1$, so the upper semicontinuity condition \cite[Theorem 12.8]{Har77} states that if $E'$ specializes to $E$, then the inequalities 
\begin{equation}\label{semicontinuity}
h^0(C, E\otimes L)\geq h^0(\PP^1, E'(\deg L))
\end{equation}
and
\begin{equation}\label{h1semicontinuity}
h^1(C, E\otimes L)\geq h^1(\PP^1, E'(\deg L))
\end{equation}
hold for all line bundles $L$ on $C$. These inequalities are in fact equivalent to each other, because the flatness of $\mathcal{E}$ over $\Delta$ implies $\chi(C,E\otimes L)=\chi(\PP^1, E'(\deg L))$.

In the case $C\cong C'\cong \PP^1$, it is well-established that these are the only obstructions. By the Birkhoff-Grothendieck theorem, 
vector bundles of degree $d$ and rank $r$ on $\PP^1$ are in bijection with \emph{splitting types}, which are collections of integers 
$
(d_1,\ldots,d_r)
$ 
with $d_1\geq d_2\geq \dots \geq d_r$ and $\sum_{1\leq i\leq r}d_i=d$. In this case, the semicontinuity condition \ref{semicontinuity} reduces to statement that if the splitting type $(d_1',\ldots, d_r')$ specializes to $(d_1,\ldots,d_r)$, then we have 
\[
\sum_{1\leq i\leq k}d_i'\leq \sum_{1\leq i\leq k} d_i
\]
for all $k$. By Theorem 14.7 of \cite{EH16},
  this semicontinuity condition is  sufficient for a vector bundle with splitting type $(d_1',\ldots, d_r')$ to specialize to one with splitting type $(d_1,\ldots,d_r)$.

We will analogously show that these are the only obstructions for any $C$ if we have $C'\cong \PP^1$.
\begin{theorem}\label{treesMain}
In the setup of Question \ref{basicQuestion}, a vector bundle $E'$ on $C'\cong \PP^1$ specializes to $E$ if and only if $E$ and $E'$ have the same degree and rank and the semicontinuity condition (\ref{semicontinuity}) holds for all line bundles on $C$.
\end{theorem}
To prove Theorem \ref{treesMain}, we ideally would construct the desired vector bundle on the family $\OC$ as an iterated extension of line bundles on $\OC$. However, this is not always possible, as the following example illustrates.
\begin{example}\label{keyExample}
Let $C$ be the union of two smooth projective curves intersecting at a node. Let $\OO(a,b)$ denote the line bundle on $C$ that has degree $a$ on the first $\PP^1$ and degree $b$ on the second. Then the vector bundle $E:=\OO(2,0)\oplus \OO(0,2)$ is allowed by semicontinuity to be the limit of the vector bundle $E':=\OO(1)\oplus \OO(3)$, but $E$ has no line subbundle of degree 3.
\end{example}
To overcome this obstacle, in the proof of Theorem \ref{treesMain} we will replace $C$ with a larger tree of smooth rational curves $C'$, such that the desired vector bundle $\mathcal{E}$ on a family $\OC'$ with special fiber $C'$ can be constructed as a deformation of an iterated extension of line bundles. We then produce the desired specializations by blowing down $\OC'$.

\begin{remark}
Deopurkar \cite{Deo17} notes that in the example above, the fact that $E$ has no line subbundle of degree $3$ means that $E'$ cannot specialize to $E$ if we impose the additional condition that the total space $\OC$ is regular. Indeed, the construction we use to prove Theorem \ref{treesMain} in this case produces a surface $\OC$ with an ordinary double point at the node of $C$.
\end{remark}

This paper is organized as follows. In Section \ref{stacksSection} we construct the moduli stacks of vector bundles on prestable curves of genus $g$ and establish their basic properties, in the process establishing that specializations of vector bundles compose (Proposition \ref{specializationsCompose}), and establish Theorem \ref{treesMain} for line bundles. Then, in section \ref{treesMainTheoremProof}, we prove Theorem \ref{treesMain}.
\subsection*{Acknowledgements}
We thank Izzet Coskun, Joe Harris, and Anand Deopurkar for helpful conversations, and Dori Bejleri for his helpful comments on a draft of this article.
\section{Preliminaries}\label{stacksSection}
\subsection{The moduli stack of vector bundles on trees of smooth rational curves}
Throughout, stacks will be stacks on $\mathit{Aff}_k$ with the fppf topology, where $k$ is an algebraically closed field, following the convention of \cite{dJHS11}, on which we rely.

Recall that a \emph{prestable} curve is a nodal connected curve. Let $\mathit{Curves}_{ps,g}$ denote the moduli stack of prestable curves of arithmetic genus $g$, in the sense of \cite[Definition 0E6T]{Stacks}. This stack is algebraic and smooth and is an open substack of the moduli stack of all genus $g$ curves (\cite[0E6U, 0E6W]{Stacks}). In addition, we will require the fact that $\mathit{Curves}_{ps,g}$ is quasi-separated and locally of finite presentation over $k$, which follows from \cite[Lemma 0DSS]{Stacks}.

Let $X_{g,r}$ be the stack over $k$ of pairs $C,E$, where $C$ is a connected nodal curve of arithmetic genus $g$ and $E$ is a vector bundle on $C$ of rank $r$. More precisely, $X_{g,r}$ is the fibered category over $\mathit{Aff}_k$ with objects over the scheme $S/k$ given by $(\pi:\mathcal{C}\ra S, E)$ with $\pi$ a flat morphism from an algebraic space $\mathcal{C}$ which is proper, locally of finite presentation, and nodal of relative dimension 1 with genus $g$ fibers such that $\pi_*(\OO_\OC)\cong\OO_S$, and $E$ is a vector bundle on $\OC$, and with morphisms coming from the  obvious maps of pairs. This stack admits a map 
\[
F:X_{g,r}\ra \mathit{Curves}_{ps,g}
\]
forgetting the data of the vector bundle. We have the following.
\begin{proposition}[\cite{dJHS11}, Proposition 3.6]
$X_{g,r}$ is an algebraic stack with locally finitely presented, separated diagonal.
\end{proposition}
\begin{proof}
We observe that the moduli stack of vector bundles of rank $r$ on curves over $k$ is just the stack $\mathit{CurveMaps}(BGL_r)$ of \cite[Definition 3.5]{dJHS11}. There is a natural map $F:\mathit{CurveMaps}(BGL_r)\ra \mathit{Curves}$, which by \cite[Proposition 3.6]{dJHS11} is representable by algebraic stacks.
 Then, since $\mathit{Curves}_{ps,g}$ is an open algebraic substack of $\mathit{Curves}$, the fiber product $\mathit{CurveMaps}(BGL_r) \times_{\mathit{Curves}} \mathit{Curves}_{ps,g}$, which is $X_{g,r}$ is an algebraic stack. Then, since $\mathit{CurveMaps}(BGL_r)$ has locally finitely presented, separated diagonal, we have its open substack $X_{g,r}$ has the same properties.
\end{proof}
We have a decomposition $X_{g,r}=\bigsqcup_{d\in \ZZ} X_{g,r,d}$, where $X_{g,r,d}$ is the moduli stack of pairs $(C,E)$ with $E$ of rank $r$ and degree $d$. Each $X_{g,r,d}$ is of course an algebraic stack with locally finitely presented, separated diagonal, but more is true.
\begin{proposition}
The forgetful map $F:X_{g,r,d}\ra \mathit{Curves}_{ps,g}$ is quasi-separated and $X_{g,r,d}$ is quasi-separated.
\end{proposition}
\begin{proof}
We must show that the diagonal map $\Delta:X_{g,r,d}\ra X_{g,r,d}\times_{\mathit{Curves}_{ps,g}}X_{g,r,d}$ is quasi-compact. The quasi-compactness of $\Delta$ is equivalent to the statement that for any family of prestable genus g curves $\pi:\OC\ra U$ over an affine base $U$ and any pair of rank $r$ degree $d$ vector bundles $E_1$, $E_2$ on $\OC$ the scheme $\mathrm{Isom}_U(E_1,E_2)$ is quasi-compact. But $\pi$ is finitely presented and $E_1$, $E_2$ are flat over U, so by \cite[Proposition 08K9]{Stacks}, $\mathrm{Isom}_U(E_1,E_2)$ is affine and hence quasi-compact.

We have that $\mathit{Curves}_{ps,g}$ is quasi-separated, so $X_{g,r,d}$ is quasi-separated because $F$ is.
\end{proof}

This result allows us to give a modified form of Question \ref{basicQuestion}.
\begin{question}\label{stackyQuestion}
Let $p$ be a $k$-point of $X_{g,r,d}$ whose associated pair $(C,E)$ has automorphism group $G$. Let $Z_p$ be the substack $p/G$ of $X_{g,r,d}$ corresponding to this point. Which points $Z_{p'}$ are in the closure of $Z_p$?
\end{question}
Question \ref{stackyQuestion} has one advantage over Question \ref{basicQuestion}, in that it is clear that if $Z_{p''}$ is in the closure of $Z_{p'}$ and $Z_{p'}$ is in the closure of $Z_p$, then $Z_{p''}$ is in the closure of $Z_p$. But its connection to Question \ref{basicQuestion} is not obvious, and is provided by the following result.
\begin{lemma}\label{technicalSpecializationLemma}
Let $(C,E)$ be a pair associated to a $k$-point $p$ of $X_{g,r,d}$. If the point $p'$ corresponding to the pair $(C',E')$ is in the closure of $Z_p$, then there is a flat family of curves $\OC\ra B$ over an irreducible affine base curve $B\ni 0$ and vector bundle $\mathcal{E}$ on $\OC$ with fiber $(C',E')$ over $0$ and fiber $(C,E)$ over any nonzero $k$-point of $B$.
\end{lemma}
\begin{proof}
Let $p$ be a $k$-point of $X_{g,r,d}$ and given a smooth atlas $V\ra X_{g,r,d}$, let $U\subset V$ be a finitely presented affine open subscheme containing some preimage $\tilde{p}$ of $p$. Suppose $p$ is in the closure of some other $k$-point $q$. We construct the fiber product.
\[
\begin{tikzcd}
F\arrow[r] \arrow[d]& U\arrow[d]\\
q\arrow[r] &X_{g,r,d}
\end{tikzcd}.
\]
Without further assumptions, $F$ may be an algebraic space, so let $\tilde{F}\ra F$ be an \'etale cover by a scheme, and let $f:\tilde{F}\ra U$ be the resulting composition. Since $p$ is in the closure of $q$, a lift $\tilde{p}$ of $p$ in $U$ will be in the closure of the image of $f$. Moreover, $f$ is finitely presented since $X_{g,r,d}$ has finitely presented diagonal. So $f$ has constructible image in $U$. $\tilde{p}$ is in the closure of this image, and in particular there is some integral locally closed subscheme $Z$ of $U$ contained in the image of $f$ such that we have $\tilde{p}\in \overline{Z}$. Then the map $\overline{Z}\ra X_{g,r,d}$ corresponds to a family of vector bundles on a family of curves $\mathcal{C}_{\overline{Z}}\ra \overline Z$ with fiber over a $k$-point of $Z$ the pair $q$ and fiber over $\tilde{p}$ the pair $p$. We have that $\overline{Z}$ is an integral affine scheme of finite type over a field, that is, an affine variety. Then there is some affine curve $B$ in $\overline{Z}$ whose set-theoretic intersection with $\overline{Z}\setminus Z$ is exactly $\tilde{p}$. The family $\mathcal{C}\ra B$ associated to $B$ satisfies all the properties in the proposition.
\end{proof}
Lemma \ref{technicalSpecializationLemma} is a key step in producing compositions of specializations. We will apply it to prove the following result.
\begin{proposition}\label{specializationsCompose}
If $(\PP^1,E_1)$ specializes to $(\PP^1,E_2)$ and $(\PP^1,E_2)$ specializes to $(C, E_3)$ in the sense of Question \ref{basicQuestion}, then $(\PP^1,E_1)$ specializes to $(C,E_3)$.
\end{proposition}
\begin{proof}
From the specialization hypotheses, we have that $(C,E_3)$ is in the closure of $(\PP^1,E_1)$ in $X_{0,r,d}$. Then there exists an affine curve $B$ with marked point $0$, a family of curves $\pi:\OC\ra B$, and a vector bundle $\mathcal{E}$ on $\OC$ such that $\OC_0\cong C$, $\OC_p\cong \PP^1_k$ for $k$-points $p\neq 0$ of $B$, and $\mathcal{E}\vert_{\OC_0}\cong E_3$ and $\mathcal{E}\vert_{\OC_p}\cong E_1$ for all $p\neq 0$. By replacing $B$ with a normalization, we may moreover assume $B$ is regular, so the formal neighborhood of 0 in $B$ is isomorphic to $\Spec(k[[t]])$.

The map $\pi:\OC\setminus \OC_0\ra B\setminus 0$ is an \'etale-locally trivial $\PP^1$ bundle, which is Zariski-locally trivial by Tsen's theorem. And the restriction of $\mathcal{E}$ to each fiber over a nonzero closed point $p$ of $b$ is isomorphic to $E_1$. Semicontinuity then gives that the restriction of $\mathcal{E}$ to the generic fiber of $\OC\ra B$ is $E_1 \times_{k} K_B$. Restricting to a formal neighborhood of $0$ in $B$ gives the desired specialization.
\end{proof}
\subsection{Specialization of line bundles}
The only specializations we will construct directly in this paper are specializations of line bundles. These specializations will be constructed using the following.

\begin{proposition}\label{lineBundlesSpecialize}
Let $\pi:\OC\ra \Delta$ be a flat family of trees of rational curves with $\Delta$ the spectrum of a DVR $R$ over $k$ with generic point $\eta$ and closed point $p$. Suppose $C_\eta \cong \PP^1_{K_R}$. If $\OC$ is smooth, then a line bundle $L$ on $C_p$ can be expressed as the degeneration of a line bundle $\OO(d)$ on $\OC_\eta$ if and only if $\deg(L)=d$.
\end{proposition}
\begin{proof}
A line bundle $\mathcal{L}$ on $\OC$ must have the same degree over each of the fibers of $\pi$, so the requirement $\deg(L)=d$ is necessary. Suppose $\deg(L)=d$.  The closure in $\OC$ of a $K_R$ point $p\in\OC_\eta$ is a divisor of degree 1 on the fibers of $\pi$. Letting $L_1$ be the line bundle associated to this divisor, by replacing $L$ with $L\otimes (L_1\vert_{C_p})^{\otimes -d}$ and twisting $\OO(d)$ down to $\OO$, we may assume that we want to specialize $\OO_{C_\eta}$ to some degree 0 line bundle $L$. $L$ is determined by the degree of its restrictions to each $\PP^1$ contained in $C_p$, and the sum of these degrees is 0. Moreover, since $\OC$ is smooth, the Weil divisor on $\OC$ given by the class of a smooth subcurve $C'$ of $C_p$ is Cartier, and the associated line bundle $\OO(C')$ has degree $1$ on each $C''$ intersecting $C'$ and degree $-e$ on $C'$, where $e$ is the number of irreducible curves intersecting $C'$ in $C_p$. We observe then that given any two intersecting smooth subcurves $C'$ and $C''$ on $C_p$, there is a line bundle of degree $1$ on $C'$ and $-1$ on $C''$, given by 
\[
\OO(\sum_{C''' \text{ on }C''\text{ side of } C'\cap C''} C''').
\]
Restrictions of these line bundles clearly generate all line bundles of degree 0 on $C_p$, because $C_p$ is connected. The result follows.
\end{proof}

\section{Proof of Theorem \ref{treesMain}}\label{treesMainTheoremProof}
In this section, we prove Theorem \ref{treesMain}. Naively, we might hope that the specializations we desire on a family $\OC$ may be produced by iterated extensions of line bundles constructed using Proposition \ref{lineBundlesSpecialize}. As demonstrated in Example \ref{keyExample}, this is optimistic in general; though we hope that $\OO(1)\oplus \OO(3)$ specializes to $E$ in the example, $E$ evidently has no line subbundles of degree 3, and so there is no hope of describing the desired vector bundle on a family as an extension of line bundles.

Instead, we will show that such an extension can be constructed if $C$ is replaced by an \emph{enlargement} $C'$, defined as a surjective map of trees of smooth rational curves $f:C'\ra C$ such that the restriction of $f$ to any irreducible component of $C'$ is an isomorphism or a constant map. Given an enlargement $C'$ of $C$ and a vector bundle $E$ on $C$, we will also call the pair $(C',f^*(E))$ an enlargement of $(C,E)$.

Allowing for enlargements will turn out to eliminate the obstruction present in Example \ref{keyExample}. The following lemma presents a way to construct a maximal-degree line subbundle on an enlargement of $(C,E)$.
\begin{lemma}\label{treesMainLemma}
Let $C$ be a tree of smooth rational curves and $E$ a vector bundle on $C$ such that $H^0(C,E(L))\neq 0$ for all line bundles $L$ of degree 0 on $C$, but $H^0(C,E(L))=0$ for some line bundle $L$ of degree $-1$. Then there is an enlargement $(C',E')$ of $(C,E)$ such that $E'$ has a line subbundle of degree at least 0.
\end{lemma}
\begin{proof}
In what follows, a \emph{coconnected subtree} of $C$ is defined to be a union $C_1$ of irreducible components in $C$ such that $C_1$ and the complement of $C_1$ in $C$ is connected. 

We proceed by induction on the number of irreducible components of $C$. If $C$ has one component, then $C\cong \PP^1$ and the lemma is trivial. Otherwise, let $L_{-1}$ be a line bundle on $C$ of degree $-1$ such that $E\otimes L_{-1}$ has no global sections. Set $E':=E\otimes L_{-1}$. Let $S$ be the set of coconnected subtrees $C_1$ of $C$ such that $H^0(C_1,E'\vert_{C_1}\otimes L_0)\neq 0$ for all line bundles $L_0$ of degree 0 on $C_1$. 

We first show that for any coconnected subtree $C_1$ of $C$, either $C_1$ or its complement $C_2:=\overline{C\setminus C_1}$ is in $S$. We proceed by contradiction. Suppose that both $C_1$ and $C_2$ are not in $S$.  Let $L_1$ and $L_2$ be degree 0 line bundles on each of $C_1$ and $C_2$ such that $H^0(C_1,E'\vert_{C_1}\otimes L_1)=H^0(C_2,E'\vert_{C_2}\otimes L_2)=0$, let $L_0$ be the line bundle on $C$ that restricts to $L_1$ and $L_2$ on $C_1$ and $C_2$ respectively, and let $L$ be the line bundle that has degree 1 on the irreducible component of $C_1$ that contains $C_1\cap C_2$, and otherwise has degree 0. Then we have an exact sequence
\[
0\ra H^0(C_1,E'\vert_{C_1}\otimes L_1)\ra H^0(C,E'(L_0\otimes L))\ra H^0(C_2,E'\vert_{C_2}\otimes L_2)
\]
and $H^0(C,E'(L_0\otimes L))=0$, contradicting our hypothesis on $E$. So either $C_1$ or $C_2$ must be in $S$.

Now let $C_1$ be an irreducible component of $C$. Let $C_2$ be an irreducible component of $C$ intersecting $C_1$ such that the coconnected subtree of $C$ containing $C_1$ but not $C_2$ is in $S$, if such a curve exists. Repeat to produce $C_3$, $C_4$ and so on. 

One of the following cases then occur.
\begin{enumerate}
    \item This process never terminates. In this case, we must have that for some $i$ this process ``turns around'' after producing $C_i$, in the sense that $C_{i-1}=C_{i+1}$.  We then have that both the coconnected subtree $C_1$ containing $C_{i-1}$ but not $C_i$ and its complement $C_2:=\overline{C\setminus C_1}$ are in $S$.
    \item This process terminates at some $C_i$. Then we have that every connected component of $\overline{C\setminus C_i}$ is in $S$.
\end{enumerate}
In the first case, noting  $C\notin S$ by the definition of $E'$, possibly after performing suitable enlargements on $C_1$ and $C_2$ we have that by the inductive hypothesis there  are line subbundles $L_1$ and $L_2$ of $E'\vert_{C_1}$ and $E'\vert_{ C_2 }$ respectively, each of degree at least 0. 
Enlarge $C$ by inserting a $\PP^1$ at the intersection $C_1\cap C_2$; call the new pair $(\tilde{C},\tilde{E'})$. Then there is a degree at least $-1$ line subbundle of $\tilde{E'}$ that restricts to $L_1$ and $L_2$ on $C_1$ and $C_2$ respectively, and restricts to $\OO(-1)$ on the curve between them. After twisting by the enlargement of $L_{-1}^*$, this produces a degree at least 0 line subbundle of $\tilde{E}$.

In the second case, let $L$ be the line bundle of degree 1 on $C_i$ and degree 0 otherwise. Then $E'(L)$ has a global section $s$ which can only vanish on a union of coconnected subtrees $C_1\cup \ldots \cup C_k$ of $\overline{(C\setminus C_i)}$. If $k=0$, we are done, so suppose $k\geq 1$. Let $C'$ be the subtree of $C$ on which $s$ is generally nonvanishing--we have that $s$, by vanishing at $k$ points of $C'$, gives a line subbundle $L'$ of degree at least $k$ of $E'(L)\vert_{C'}$. Moreover, by the inductive hypothesis and the fact that every $C_j$ is in $S$, there are nonnegative degree line subbundles $L_j$ of $E'\vert_{C_j}$ for each $j$, Enlarge $C$ to a tree $\tilde{C}$ by inserting a $\PP^1$ at each point $C_j\cap C'$. There is a line subbundle $\tilde{L}$ of $\tilde{E}$ on $\tilde{C}$ that restricts to $L_j$ on each $C_j$, restricts to $L'$ on $C'$, and restricts to $\OO(-1)$ on each new $\PP^1$. Then $\tilde{L}$ is a line subbundle of $\tilde{E}$ of nonnegative degree.
\end{proof}
\begin{remark}
In fact, under the conditions of the lemma above the line subbundle $L$ must have degree exactly 0. For if $L$ has positive degree, all degree $-1$ twists of it would have sections, so $E'$ would have sections.
\end{remark}
\begin{example}
Let $C$ be a tree of two smooth rational curves intersecting in a node, and let $E:=\OO(2,0)\oplus \OO(0,2)$ be as in Example \ref{keyExample}. We have that every degree $-3$ twist of $E$ has global sections, but $E':=E\otimes \OO(-2,-2)\cong \OO(0,-2)\oplus \OO(-2,0)$ has no global sections. 
 The set $S$ associated to $E'$ in the proof of Lemma \ref{treesMainLemma} then contains every coconnected subtree of $C$ except $C$ itself; so case (1) of the proof applies. Enlarging $C$ by inserting a smooth rational curve at the node, we then have that the line bundle of degree 0 on the two original components of $C$ and degree $-1$ on the new component is a line subbundle of the enlargement of $E'$, as we wanted. 
\end{example}
\begin{lemma}\label{extComputation}
Let $\pi:\OC\ra \Delta$ be a family of trees of smooth rational curves over $\Delta$ an affine $k$-schme, and let $\mathcal{E}$ and $\mathcal{G}$ be vector bundles on $\OC$. Let $p$ be a $k$-point of $\Delta$, let $C_p$ be the fiber of $\pi$ over $p$, and let $F$ be a vector bundle on $C_p$ given as an extension
\begin{equation}\label{plhldr}
0\ra E\ra F\ra G\ra 0,
\end{equation}
where $E=\mathcal{E}\vert_{C_p}$ and $G=\mathcal{G}\vert_{C_p}$. Then there exists an extension of vector bundles on $\OC$,
\[
0\ra \mathcal{E}\ra \mathcal{F}\ra \mathcal{G}\ra 0,
\]
that restricts to (\ref{plhldr}) on $C_p$.
\end{lemma}
\begin{proof}
We must show that the map
\[
\mathrm{res}:\mathrm{Ext}^1(\mathcal{G},\mathcal{E})\ra \mathrm{Ext}^1(G,E)
\]
given by restricting vector bundles to $C_p$ is surjective. To see this, we use the diagram
\[
\begin{tikzcd}
\mathrm{Ext}^1(\mathcal{G},\mathcal{E}) \arrow[r, "\mathrm{res}"] \arrow[d, "\cong"]
&\mathrm{Ext}^1(G,E) \arrow[d,"\cong"]\\
H^1(\mathcal{C},\mathcal{G}^*\otimes \mathcal{E}) \arrow[r]& H^1(C,G^*\otimes E),
\end{tikzcd}
\]
where the bottom map is the map on cohomology coming from the restriction of $\mathcal{G}^*\otimes \mathcal{E}$ to $C_p$. This diagram is commutative
, and the bottom map is a surjection because $H^2(\mathcal{C}, \mathcal{G}^*\otimes \mathcal{E}\otimes \mathcal{I}_p)=0$ since $\pi$ has fiber dimension 1 over an affine base.
\end{proof}
\begin{remark}
This lemma gives no control over what the restriction of $\mathcal{F}$ to $\pi^{-1}(\Delta\setminus p)$ looks like. In practice, when we use the lemma, we will assume that there are no nonsplit extensions of $\mathcal{G}$ by $\mathcal{E}$ on $\pi^{-1}(\Delta\setminus p)$, so the restriction of $\mathcal{F}$ to $\pi^{-1}(\Delta\setminus p)$ will be
\[
\mathcal{E}\vert_{\pi^{-1}(\Delta\setminus p)}\oplus \mathcal{G}\vert_{\pi^{-1}(\Delta\setminus p)}.
\]
\end{remark}

We now have every tool we need to prove Theorem \ref{treesMain}. It will be a corollary of the following result.
\begin{theorem}\label{treesMainEnlarged}
Suppose $E$ is a vector bundle on $C$ and $E'$ a vector bundle on $\PP^1$  such that $E$ and $E'$ have the same degree and rank and the semicontinuity condition \ref{semicontinuity} holds for every line bundle $L$ on $C$. Then there exists an enlargement $f:C'\ra C$ of $C$, a specialization $E''$ of $E'$ on $\PP^1$, a flat family $\mathcal{C}\ra \Delta$ over $\Delta=\Spec(k[[t]])$ with regular total space, general fiber isomorphic to $\PP^1_{k((t))}$, and special fiber isomorphic to $C'$, and a vector bundle $\mathcal{E}$ on $\mathcal{C}$ that restricts to $f^*(E)$ on the special fiber and to $E''$ on the generic fiber.
\end{theorem}
\begin{proof}
We first note that given any tree of smooth rational curves $C'$, it is easy to construct a family of trees of smooth rational curves $\pi:\OC\ra \Spec(k[[t]])$ with generic fiber $\PP^1$ and special fiber $C'$ such that $\OC$ is regular. Starting with the trivial family $\PP^1\times \Spec(k[[t]])$, the desired family can be constructed by iteratively blowing up points on the central fiber. 

To construct the vector bundle $\mathcal{E}$ on this regular total space $\OC$, we proceed by induction on the rank of $E$. For $E$ a line bundle, the result follows from Proposition \ref{lineBundlesSpecialize}. Now suppose that $E$ and $E'$ are vector bundles of rank at least 2 satisfying the hypotheses of the theorem. Suppose $d$ is maximal such that $H^0(C,E(L_{-d}))\neq 0$ for all line bundles $L_{-d}$ of degree $-d$ on $C$. Then by Lemma \ref{treesMainLemma},
 possibly after replacing $C$ with an enlargement $f:C'\ra C$, the vector bundle $f^*(E)$ has a line subbundle $L_d$ of degree $d$, while by semicontinuity no direct summand of $E'$ has degree greater than $d$. Let  $E''$ be the unique vector bundle that satisfies
\[
h^0(\PP^1,E''(e))=\max(h^0(\PP^1, \OO(d+e)), h^0(\PP^1, E'(e)))
\]
for all $e$. Clearly, $E''$ is a specialization of $E'$. Moreover, $E''$ satisfies the semicontinuity property (\ref{semicontinuity}) with respect to $f^*(E)$ for all choices of line bundles, because for any line bundle $L$ we have $h^0(C, E\otimes L)\geq h^0(\PP^1, E'(\deg L))$ by hypothesis, and we have the chain of inequalities
\[
h^0(C', f^*(E)\otimes L)\geq h^0(C', L_d\otimes L)\geq h^0(\PP^1, \OO(d+\deg L))
\]
that are satisfied for any $L$.  We now show that some specialization of $E''$ on $\PP^1$ specializes to an enlargement of $f^*(E)$. Let $Q':=E''/\OO(d)$. We have that $Q'$ has the same degree and rank as $Q:=f^*(E)/L_d$, and we have
\[
h^0(\PP^1, Q'(e))=\twopartdef{h^0(\PP^1, E''(e))-d-1-e}{ e\geq -d-1}{0}{e<-d-1},
\]
so $Q$ satisfies the semicontinuity condition (\ref{semicontinuity}) with respect to $Q'$ for all line bundles. So by the induction hypothesis there is an enlargement $f':C''\ra C'$ and a vector bundle $\mathcal{Q}$ on a smoothing $\mathcal{C}\ra \Delta$ of $C''$ that restricts to some specialization $Q''$ of $Q'$ on the general fiber and $Q$ on the special fiber. Let $\mathcal{L}_d$ be the line bundle on $\OC$ that restricts to  $f'^*(L_d)$ on the special fiber. By Lemma \ref{extComputation}, there is then a vector bundle $\mathcal{E}$ on $\mathcal{C}$ that is given by an exact sequence
\[
0\ra \mathcal{L}_d\ra \mathcal{E}\ra \mathcal{Q}\ra 0
\]
and restricts to the exact sequence
\[
0\ra f'^*(L_d)\ra f'^*(f^*(E))\ra f'^*(Q)\ra 0
\]
on $C''$. Such an extension automatically restricts to $\OO(d)\oplus Q''$ on the general fiber, because there are no nontrivial extensions of $Q''$ by $\OO(d)$ on $\PP^1$. $\OO(d)\oplus Q''$ is a specialization of $E''$ and therefore of $E$, so $\mathcal{E}$ is the vector bundle giving the desired specialization. 
\end{proof}
\begin{proof}[Proof of Theorem \ref{treesMain}]
By Theorem \ref{treesMainEnlarged}, we have that there exists a specialization $E''$ of $E'$ on $\PP^1$, an enlargement $(\tilde{C},\tilde{E})$ of $(C,E)$, a flat family $\tilde{\mathcal{C}}\ra \Spec(k[[t]])$ with regular total space, and a vector bundle $\tilde{\mathcal{E}}$ on $\tilde{\mathcal{C}}$ such that $C_\eta\cong \PP^1_{k((t))}$, $C_p\cong \tilde{C}$, and the restriction of $\tilde{\mathcal{E}}$ to $C_\eta$ and $C_p$ is $E''$ and $\tilde{E}$ respectively.  The enlargement map $f:\tilde{C}\ra C$ extends to a blow-down map $\phi:\tilde{\OC}\ra \OC$ which contracts every $\PP^1$ subtree contracted by $f$. Moreover $\tilde{E}$ is trivial on the components of $\tilde{C}$ contracted by $f$, so likewise $\tilde{\mathcal{E}}$ is trivial on such components. Then the direct image $\phi_*(\tilde{\mathcal{E}})$ is a vector bundle on $\OC$ giving a specialization from $E''$ to $E$. Then, by Proposition \ref{specializationsCompose}, we have that there exists a family of trees of smooth rational curves $\pi:\OC\ra \Delta$  and a vector bundle $\mathcal{E}$ on $\OC$ that restricts to $(\PP^1,E')$ on the generic fiber and to $(C,E)$ on the special fiber, as we were to show.

\end{proof}
\bibliographystyle{plain}
\bibliography{squirt}
\end{document}